\documentclass[11pt]{amsart}
\usepackage{amsmath}
\usepackage{amsthm}
\usepackage{amssymb}

\DeclareMathOperator{\image}{''}

\DeclareMathOperator{\Ult}{Ult}
\DeclareMathOperator{\cf}{cf}
\DeclareMathOperator{\crit}{crit}

\DeclareMathOperator{\Col}{Col}
\DeclareMathOperator{\otp}{otp}

\newcommand{\ZFC}{\mathrm{ZFC}}
\newcommand{\Ord}{\mathrm{Ord}}

\newcommand{\GCH}{\mathrm{GCH}}
\newcommand{\Inacc}{\mathrm{Inacc}}

\newtheorem{theorem}{Theorem}
\newtheorem{lemma}[theorem]{Lemma}
\newtheorem{claim}[theorem]{Claim}
\newtheorem{remark}[theorem]{Remark}
\newtheorem{definition}[theorem]{Definition}
\newtheorem{conjecture}[theorem]{Conjecture}

\newtheorem{question}[theorem]{Question}
\newtheorem{corollary}[theorem]{Corollary}
\author{Eilon Bilinsky}
\address{Institute of Mathematics,
 The Hebrew University of Jerusalem,
 Jerusalem 91904, Israel}
\email{eilonbil@mail.tau.ac.il}

\author{Yair Hayut}
\address{Einstein Institute of Mathematics,
 The Hebrew University of Jerusalem,
 Jerusalem 91904, Israel}
\email{yair.hayut@mail.huji.ac.il}
\thanks{The second author was supported by the Israel Science Foundation, grant number 1967/21}
\begin{document}
\title{The Spectra of transitive models}
\maketitle
\begin{abstract}
In this paper we study the spectrum of heights of transitive models of theories extending $V = L[A]$, under various definitions. In particular, we investigate the consistency strength of making those spectra as simple as possible.
\end{abstract}
\section{Introduction}
From a Platonist perspective, some ordinals are more definable then others. For example, $\omega_1^L$ is the maximal ordinal such that there is a transitive model $M$ satisfying $V=L$ and every ordinal is countable. Indeed, this model, $M = L_{\omega_1^L}$ is rather absolute and satisfies various definitions of maximality, with respect to the theory above. This observation remains true in the generic multiverse. 

In this paper we are trying to formalize and analyse this property. In particular, we ask what is the class of ordinals from which there is a theory $T$ that \emph{pins} them in the sense that they are the maximal height of a transitive model of $T$, perhaps in an absolute way (see the definitions in Section \ref{section:definitions}).

This problem is related to a wider model theoretical program. Recall that Morley's Theorem states that for any countable theory $T$ and a type $p$, there are models of $T$ that omit $p$ of arbitrarily large cardinality, if and only if the possible cardinalities of such models are unbounded below $\beth_{\omega_1}$, \cite{Morley}. Shelah has generalized this statement and showed that there is a connection between the cardinalities of models that satisfy a theory $T$ and omits a type $p$ and the supremum of \emph{definable orders} which can be guaranteed to be well founded in $\omega$-models of a related theory\footnote{an $\omega$-model is a model of an extension of the language of set theory for which $\omega$ is standard}. 
This correspondence extends to uncountable languages, allowing non-trivial generalizations of Morley's theorem, see \cite{BarwiseKunen}. 

In some sense, the main motivation of this paper, is dual to Shelah's observation. Given an ordinal, can we pin it down as the maximal height of transitive models of certain theory? This question can be formulated in a many non-equivalent ways, and in this paper we studied a couple of related formulations, and their connections.  

Let us mention that the problem of characterizing the possible spectrum of $\omega$-models is related to a conjecture  of Shelah: Let $T$ be a theory such that $T$ has $\omega$-model of cardinality $\aleph_{\omega_1}$. Then $T$ has a model of cardinality continuum. One of the motivations of this work is to study certain aspects of the variant of this problem in which we replace $\omega$-correctness by transitivity. 

\section{Definitions}\label{section:definitions}
We are interested in trying to identify ordinals which can be pinned down as \emph{the} height of transitive models of a certain theory. We would like those notions to be relatively absolute, and we will investigate several notions of uniqueness for the transitive models. 

All our theories $T$ are over the language $\{=, \in, A\}$ where $A$ is a unary relation. Moreover, {\bf we always assume that $T$ includes the axiom $V = L[A]$}. 

The inclusion of the axiom $V = L[A]$ is intended to ensure a (strong) version of the axiom of choice. In particular, we want to rule out "wide" models, i.e.\ models of cardinality much larger than their height, as in \cite{Friedman75}. Moreover, this allows one to avoid the need of coding the structure of $L$ on the ordinals of the model. 

While this assumption is at some points unnatural, we keep it in order to make the definitions ahead compatible with the intuition that the von-Neumann ordinals capture the order type of the well orders of transitive models, without assuming that our models satisfy any strong version of replacement. Nevertheless, in our models, Mostowski collapsing lemma might still fail. 

Many of the results of this paper can be proven with the same argument without this restriction, but e.g.\ Theorem \ref{thm:S-subset-of-Sp-max} and Lemma \ref{lemma:omega_1-not-in-Sp*} rely on it in an essential way. 
\begin{definition}
Let $S$ be the class of all ordinals $\alpha$ such that there is a theory $T$ such that $\alpha$ is the maximal ordinal such that there is a model $M \models T$, with $M \cap \Ord = \alpha$. 

Let $S_u$ be class of all $\alpha$ such that there is a transitive model $M$, $M \cap \Ord = \alpha$, and there is no $N \neq M$ transitive, $N \cap \Ord \geq M \cap \Ord$, $N \equiv M$. 
\end{definition}

\begin{definition}\label{definition:S*}
Let $S^*$ be the collection of all $\alpha$ such that there is a theory $T$, and a transitive model $M$ of $T$ such that $M \cap \Ord = \alpha$ and in any generic extension there is no $N \models T$, transitive, with $N \cap \Ord > \alpha$. 

Let $S^*_u$ be the collection of all $\alpha$ such that there is a transitive model $M$ such that $M \cap \Ord = \alpha$ and in any generic extension there is no $N \neq M$ transitive, $N \cap \Ord \geq M \cap \Ord$, $N \equiv M$.
\end{definition}
Clearly, $S^* \subseteq S$, $S_u^* \subseteq S_u \subseteq S$.
\begin{remark}
Every countable ordinal is in $S^*_u$, as we can code its order-type into the theory. In particular $\omega_1 \subseteq S^*$.
\end{remark}

\begin{definition}
Let $T$ be a theory. 

Let $\mathfrak{Sp}_T$ be the class of all ordinals $\alpha$ such that there is a transitive model $M$ of $T$, $M \cap \Ord = \alpha$ and $M$ is maximal in the sense that there is no transitive $N$, and a non-trivial elementary embedding $j \colon M\to N$.\footnote{The only trivial elementary embedding is $id \colon M \to M$, so if $N \neq M$, the embedding is automatically non-trivial.}

Let $\mathfrak{Sp}_{T,u}$ be the class of all ordinals $\alpha$ such that there is a unique transitive model $M$ of $T$, $M \cap \Ord = \alpha$ and $M$ is maximal in the sense that there is no transitive $N$, and non-trivial elementary embedding $j \colon M\to N$.

Let $\mathfrak{Sp}^*_T$ be the class of all ordinals $\alpha$ such that there is a transitive model $M$ of $T$, $M \cap \Ord = \alpha$ and in any generic extension, $M$ is maximal.

Let $\mathfrak{Sp}^*_{T,u}$ be the class of all ordinals $\alpha$ such that there is a unique transitive model $M$ of $T$, $M \cap \Ord = \alpha$ and in any generic extension, $M$ is unique and maximal.

Let $\mathfrak{Sp} = \bigcup_T \mathfrak{Sp}_T$, $\mathfrak{Sp}^* = \bigcup_T \mathfrak{Sp}^*_T$ and so on. 
\end{definition}
While $S$ is at most of cardinality $2^{\aleph_0}$, $\mathfrak{Sp}$ can be a proper class. Indeed, by Lemmas \ref{lemma:reaching-to-measurable} and \ref{lemma:S-and-Sp-below-measurable}, $\mathfrak{Sp}$ is a set if and only if there is a measurable cardinal.

\begin{claim}
For every $\alpha < \omega_1$, there is $T$ such that $\mathfrak{Sp}_{T,u}^* = \{\alpha\}$. In particular, $\omega_1 \subseteq \mathfrak{Sp}^*_u \subseteq \mathfrak{Sp}^*$.
\end{claim}

\begin{lemma}\label{lemma:S-and-Sp-below-measurable} Let $\kappa$ be measurable. $\mathfrak{Sp} \cup S \subseteq \kappa$. 
\end{lemma}
\begin{proof}
Let $M$ be a transitive model of $T$, $M\cap \Ord \geq \kappa$. Let $j \colon V \to N$ be an elementary embedding with critical point $\kappa$, $j(\kappa) > M \cap \Ord$ and $N$ transitive. Such an embedding exists, by taking a sufficiently long iteration of a measure on $\kappa$ (note that $N$ is typically not closed under countable sequences). 

Then $j\image M \prec j(M)$ and $j(M) \cap \Ord = j(M \cap \Ord) \geq j(\kappa) > M \cap \Ord$. 
\end{proof}
By applying the argument locally, we obtain the following lemma.
\begin{lemma}
Let $\kappa$ be a weakly compact cardinal. Then $[\kappa, \kappa^{+}) \cap (\mathfrak{Sp} \cup S) = \emptyset$.
\end{lemma}
\section{The properties of $S$ and $\mathfrak{Sp}$}
The study of the possible spectrum of models of $\mathcal{L}_{\omega_1,\omega}$-sentences plays an important role in model theory. For example, in the recent paper \cite{BaldwinShelah2022}, a complete $\mathcal{L}_{\omega_1,\omega}$-sentence is constructed with maximal models of arbitrary large cardinality  below a measurable cardinal. In \cite{SinapovaSouldatos2020}, an $\mathcal{L}_{\omega_1,\omega}$-sentence is constructed with spectrum that contains exactly the cardinality of branches of $\omega_1$-Kurepa trees and the continuum.

As customary, models of a $\mathcal{L}_{\omega_1,\omega}$-sentence can be viewed as models of a certain first order theory with the correct $\omega$.  

Thus, one can view our definition of $\mathfrak{Sp}_T$ as asking about the spectrum of models of $T$ that computes ordinals correctly. For those models, the height of their ordinals is a finer measurement of their size than their cardinality.\footnote{We include the assertion that $T$ contains the sentence $V = L[A]$, in order to make sure that the von-Neumann ordinals are the correct way to measure the height of the model.} 

\begin{lemma}
$\omega_1 \in S$. Moreover, for every $\alpha$, if $\alpha \in S$ then $\aleph_{\alpha}$ and $\beth_\alpha$ are in $S$.
\end{lemma}
\begin{proof}
For $\omega_1$, take $T$ to contain the statement "every ordinal is countable". 

In general, in order to see that if $\alpha \in S$ then $\aleph_{\alpha} \in S$, let us consider the non-trivial case that $\alpha < \aleph_{\alpha}$. Let $T_0$ witness $\alpha \in S$. Let $T_1$ be the theory "$A$ codes two predicates, $A_0, A_1$. $L[A_1]$ thinks that $\exists \gamma, L_\gamma[A_0] \models T_0$ and there is a maximal such $\gamma$. For every infinite cardinal $\kappa$ there is $\beta < \gamma$ such that $\otp \{\mu < \kappa \mid \mu \text{ is an infinite cardinal}\} = \beta$."

The argument for $\beth_\alpha$ is analogous.
\end{proof}

\begin{theorem}\label{thm:S-subset-of-Sp-max}
Assume that there is no inner model with a measurable cardinal. Then \[S \subseteq \{\max \mathfrak{Sp}_T \mid T \text{ is a theory and }\max \mathfrak{Sp}_T\text{ exists}\}.\]
\end{theorem}
\begin{proof}
Let us denote $\mathfrak{Sp}_{max} = \{\max \mathfrak{Sp}_T \mid T \text{ is a theory and }\max \mathfrak{Sp}_T\text{ exists}\}$, and let us show that $S\subseteq \mathfrak{Sp}_{max}$. 

First, as both sets contain $\omega_1$, we may restrict our attention to uncountable ordinals.

Let $\gamma \in S$ be uncountable, and let us assume that $M$ witnesses that, namely $M$ is a transitive model and there is no $N$ transitive with $M\equiv N$, and $N \cap \Ord > M \cap \Ord$. 

Let $T = Th(M)$. By our assumption, $M = L_\gamma[A]$. We would like to show that there is a predicate $B$, such that $L_{\gamma}[A,B]$ witnesses $\gamma \in \max \mathfrak{Sp}_{T'}$ for some $T'$. The theory $T'$ will state that $L_\gamma[A]$ is a model of $T$ and that $L_\gamma[B]$ is going to be a model of some fragment of $\ZFC$. In particular, $T'$ contains the following statements: 
\begin{itemize}
\item every regular ordinal is a cardinal,  
\item if $\zeta$ is a cardinal then $\zeta \leq \rho < \zeta^+$ if and only if there is $x\subseteq \zeta \times \zeta$ with $\otp x = \rho$,\footnote{In order for this equality to make sense, we interpret it as the existence of an order preserving bijection between $x$ and $\rho$.} and
\item if $a$ is a set of ordinals which is constructed at level $\beta$ and $\forall n < \omega \beta + n$ exists then the Mostowski collapse of $a$, $\pi_a \colon a \to \otp a$ exists.  
\end{itemize}
Assume that for every choice of $B$ this is not the case, so there is an elementary embedding $j \colon L_\gamma[A, B] \to L_{\gamma'}[A',B']$. 

Since $L_\gamma[A] \models T$, by elementarity $L_{\gamma'}[A']\models T$. As $\gamma$ is the maximal height of a transitive model of $T$, and $\gamma' \geq \gamma$, we conclude that $\gamma' = \gamma$. 

\begin{lemma}\label{lemma:B-can-code-cofinality}
For every predicate $A$, there is a predicate $B \subseteq L_\gamma$ such that $L_\gamma[A, B] \models T'$ and for every $\alpha < \gamma$ such that $\cf \alpha < \alpha$, there is an element of $L_\gamma[B]$ witnessing it. Moreover, for every $\alpha < \gamma$, there is bijection in $L_\gamma[B]$ between $\alpha$ and $|\alpha|$.
\end{lemma}
\begin{proof}
First, we may assume that $\gamma$ is a limit ordinal. Indeed, if $\gamma = \gamma' + n$ for some natural $n$ then note that the successor ordinals after $\gamma'$ are automatically singular in a definable way and there is a definable bijection between them and $|\gamma'|$. Moreover, the last item of $T'$ does not address elements that were constructed after the level $\gamma'$. So, we will just need to verify that the codinality of $\gamma'$ itself is computed correctly, and the argument for that does not differ from the argument in the case that $\gamma$ is limit. 

Let $\gamma$ be a limit ordinal. So $L_{\gamma} \supseteq \gamma \times \gamma \times \gamma$. For every singular $\alpha < \gamma$, let $g_\alpha \colon \cf \alpha \to \alpha$ be a witness for the singularity of $\alpha$. Take $B_1$ to be $\bigcup \{\alpha\}\times g_\alpha$. Similarly, take $B_2$ to introduce a bijection between $\alpha$ and $|\alpha$ for every $\alpha < \gamma$, 

In order to verify that the Mostowski collapses exist, we will define $B_3$ to be a predicate of the form $\bigcup\{\{a\}\times \pi_a \mid a\in \mathcal{P}(\gamma) \setminus \Ord\}$. Finally, take $B = B_1 \cup B_2 \cup B_3$.
\end{proof}
\begin{lemma}
Let $j \colon L_\gamma[B] \to L_\gamma[B']$ be non-trivial elementary embedding, where $B$ is as  in Lemma \ref{lemma:B-can-code-cofinality} above. 
Then, the critical point of $j$ is a weakly inaccessible cardinal.
\end{lemma}
\begin{proof}
Indeed, $\crit j$ must be regular in $L_\gamma[B]$, so it is regular in $V$. 

If $\crit j = \mu^+$ for some cardinal $\mu$, then in $L_\gamma[B]$ there are no regular cardinals between $\mu$ and $\mu^+$. By elementarity, it implies that in the target of $j$, which is $L_\gamma[B']$, there is no regular cardinal between $\mu$ and $j(\crit j) > \mu^+$. As being singular is $\Sigma_1$, we conclude that the same holds in $V$, contradicting the regularity of $\mu^+$.
\end{proof}
\begin{lemma}
Let $\kappa = \crit j$. Then $\gamma > \kappa^+$.
\end{lemma}
\begin{proof}
As before, we assume that $\gamma$ is a limit ordinal (otherwise, we will argue for the maximal limit ordinal below $\gamma$). 

Let $\kappa < \alpha < \gamma$. If $\kappa \leq \alpha < \kappa^{+}$, then there is $f \colon \kappa \to \alpha$, a bijection in $L_{\gamma}[B]$. So, $j(f) \in L_{\gamma}[B']$. As $\gamma$ is a limit ordinal, $j(f)\restriction \kappa$ is a member of $L_{\gamma}[B]$. Let $\pi \colon j\image \alpha \to \alpha$ be the Mostowski collapse, and as $j\image \alpha \in L_\gamma[B']$, and $L_{\gamma}[B']\models T'$, we conclude that $\pi \in L_{\gamma}[B']$ and thus $\pi\circ (j(f) \restriction \kappa) \in L_\gamma[B']$. In particular, $|\alpha|^{L_{\gamma}[B']} = \kappa$. As $j(\kappa)$ is a cardinal in $L_\gamma[B']$, is must be as least $\kappa^{+}$. We conclude that $\gamma > \kappa^+$. 
\end{proof}
Next, we would like to use Dodd-Jensen's covering lemma in order to derive an inner model with a measurable cardinal. Recall that if there is no inner model with a measurable cardinal then $K=K_{DJ}$ satisfies the covering lemma --- for every $X \subseteq \Ord$ there is $Y \in K$ such that $X \subseteq Y$ and $|Y| = |X| + \aleph_1^V$, \cite{DoddJensen}.

Let us pick $B$ to code $K_\gamma$. Namely, $L_{\gamma}[B]$ will satisfy $T'$ and in addition will contain a definable subclass which is $K_\gamma$. By the above lemmas, we know that $\gamma > \kappa^{+}\geq (\kappa^{+})^K$ and thus $j$ induces a $K$-ultrafilter, $U$. We need to show that the internal ultrapower, $\Ult(K, U)$, is well founded. Thus, there is non-trivial elementary embedding $\iota \colon K \to M$, and we can derive an inner model with a measurable cardinal. 

Indeed, let us assume that this is not the case. Then, there is a sequence of functions $f_n \colon \kappa \to \Ord$ such that $[f_n]_U$ is $E$-decreasing (there $E$ is the membership relation of the ultrapower). While the sequence is not in $K$, it is covered by a set of size $\aleph_1^V$ in $K$, by the covering lemma. 

Let $\{g_\alpha \mid \alpha < \rho\}$ be a covering set, with $|\rho| \leq \aleph_1^V$, so $\rho < \kappa$. Let $X_{\alpha,\beta} = \{\rho < \kappa \mid g_\alpha(\rho)\in g_\beta(\rho)\}$, and $\mathcal{X} = \{X_{\alpha,\beta} \mid \alpha \neq \beta\}$. 
The set $\mathcal{X}$ is a member of $K$ and its cardinality is strictly below $\kappa$. By acceptability, it belongs to $K_{\gamma}$ (as $\gamma > (\kappa^{+})^K$). 
Moreover, the collection $\{(\alpha, \beta) \in \rho^2 \mid \kappa \in j(X_{\alpha,\beta})\}$ belongs to $L_{\gamma}[B']$, as $\gamma$ is a limit ordinal, and thus to its version of $K$, that we will denote by $K'$. 
As $K_\gamma$ satisfies acceptability, by elementarity so does $K'$. So, this set must belong to $K'_\zeta$ for some $\zeta < \kappa$. 
Since $\crit j = \kappa$, $K'_\zeta = K_\zeta$ and thus this set belongs to $K_\gamma$.

Take \[\mathcal{Y} = \{X_{\alpha, \beta} \mid X_{\alpha, \beta}\in U\} \cup \{\kappa\setminus X_{\alpha,\beta} \mid X_{\alpha,\beta}\notin U\}.\] By the reasoning above, $\mathcal{Y} \in K$ and its intersection must be non-empty. But this is a contradiction, as any element $\delta$ in the intersection must satisfy that $f_n(\delta)$ is an infinite decreasing sequence of ordinals.
\end{proof}
\begin{question}
Can Theorem \ref{thm:S-subset-of-Sp-max} be proved without the anti-large cardinal hypothesis?
\end{question}
Let us add a couple of remarks on $\mathfrak{Sp}$. 
\begin{lemma}\label{lemma:reaching-to-measurable}
There is a recursive theory $T$ such that $\mathfrak{Sp}_T$ is bounded if and only if there is a measurable cardinal, and in this case, $\sup \mathfrak{Sp}_T$ is the first measurable.  
\end{lemma}
\begin{proof}
Let us present first the proof under the hypothesis of $\GCH$, since the theory in this case is slightly more natural. 
Let $T$ be $\ZFC^-$ together with collection and the assertion that there is a maximal cardinal, $\mu$, and every set $x$ is of cardinality $\leq \mu$ (note that $H(\mu^+)\models T$ for all $\mu$ infinite). 

Let $\kappa$ be the least measurable, if there is one, and otherwise $\Ord$. For every cardinal $\rho < \kappa$, $M = H(\rho^+)$ is a maximal model of $T$. Otherwise, there is some transitive model $N$ and an elementary embedding $j \colon M \to N$. Clearly, $j$ has a critical point (as otherwise, $|\rho^{+}|^N \leq \rho$). Let $\zeta = \crit j$. Then, since $\mathcal{P}(\zeta) \subseteq M$, $\zeta$ is measurable.

We conclude that every successor cardinal below $\kappa$ belongs to $\mathfrak{Sp}_T$.

Let us deal with the general case. Let us look at the theory of $L_{\alpha}[A]$ where for some $\rho < \alpha$, $A$ codes an enumeration of $\mathcal{P}(\rho)$ in order type $\alpha$ and this model satisfies Separation and that the power set of every $\zeta < \rho$ exists. Let us assume that this model is non-maximal. Then, there is an elementary embedding $j \colon L_{\alpha}[A] \to L_{\beta}[\tilde{A}]$. 

First, this embedding must have a critical point. Otherwise, the elementary embedding $j$ "stretches" the enumeration given by $A$, and introduce a new subset of $\rho$. Indeed, $j(\rho) > \rho$.

Thus, we can obtain a measure on $\kappa = \crit j$, as $\mathcal{P}(\kappa) \subseteq L_\alpha[A]$.   
\end{proof}

The phenomena that was obtained for a measurable cardinal is by no means unique. Let us first define the general scheme. 
\begin{definition}
A theory $T$ is \emph{reaching} if $\mathfrak{Sp}_T$ is non-empty and does not contain a maximal element. 
\end{definition}
The previous theorem shows that there is a theory $T$ such that $\ZFC$ proves that $T$ is reaching and $\mathfrak{Sp}_T$ is either unbounded or $\sup \mathfrak{Sp}_T$ is the first measurable. 
\begin{lemma}
Let $\Phi(x)$ be a property which holds only for limit cardinals and if $V \models \Phi(\kappa)$ and $M$ is a transitive model of some fragment of $\ZFC^-$, then $M \models \Phi(\kappa)$. 

Then, there is a reaching theory $T_\Phi$ such that either $\mathfrak{Sp}_T$ is unbounded 
or $\sup \mathfrak{Sp}_T = \kappa$, where $\kappa$ is the least cardinal such that $\Phi(\kappa)$.
\end{lemma}
\begin{proof}
Let $T$ be the statement that the corresponding fragment of set theory holds, there is a maximal cardinal and there is no $\alpha$ such that $\Phi(\alpha)$ holds. 
\end{proof}

\begin{corollary}
There is a reaching theory $T_{inacc}$ such that $\sup \mathfrak{Sp}_T$ is the least inaccessible if there is one, and unbounded otherwise. Similarly for the least Mahlo, greatly Mahlo and reflecting cardinal.
\end{corollary}
\begin{question}
Is there a theory $T$ such that  $\ZFC$ proves that it is reaching to the least weakly compact cardinal?
\end{question}
\section{Upwards absolute versions}
We are now move to deal with the upwards absolute versions of $S$ and $\mathfrak{Sp}$, $S^*, \mathfrak{Sp}^*$, (see Definition \ref{definition:S*}). 
\begin{lemma}
There is a theory $T$ such that if $\omega_1 \notin \mathfrak{Sp}^*_T$ then $\omega_1$ is weakly compact in $L$.
\end{lemma}
\begin{proof}
Let $T$ be the theory containing the statement: $A$ codes a binary tree, $\mathcal T$, of height the ordinals and for every $\eta \in \mathcal T$, $\eta$ is in $L$. Moreover for every infinite ordinal $\gamma$, $A$ codes a surjection from $\omega$ onto $\gamma$. Let us assume that there is no $M$ transitive such that $M \models T$, $M \cap \Ord = \omega_1$ and $M$ is generically maximal. Pick a binary tree $S$ of height $\omega_1^V$ in $L$, such that all its levels are of cardinality $<\omega_1^V$ in $L$, and take $A$ to code $S$ as well as surjections from $\omega$ to each infinite ordinal. As the model $M = L_{\omega_1^V}[A]$ satisfies $T$, it cannot be generically maximal. 

So, in some generic extension there is a transitive model $N$ such that $\langle L_{\omega_1^V}, \in, S\rangle \prec N$. Indeed, any elementary embedding $j \colon M \to N$ cannot move any ordinal in $M$ (as they are countable), so as $j$ is non trivial, $N \neq M$. By the same arguments as before, we conclude that $N \cap \Ord > M \cap \Ord = \omega_1^V$.
Moreover, $S^N$ is a tree of height $N \cap \Ord$. There is $\eta \in N$ such that $\eta$ is of level $\omega_1^V$. By elementarity, $S^N \restriction \omega_1^V = S$ and in particular, $\eta$ is a cofinal branch, and belongs to $L^N$. As $N$ is transitive, it belongs to $L$. 

So, every $\omega_1^V$-tree in $L$ has a branch in $L$, and thus it is weakly compact in $L$.  
\end{proof}

\begin{lemma}\label{lemma:omega_1-not-in-Sp*}
Let $\kappa$ be weakly compact. Then, in the generic extension by the Levy collapse $\Col(\omega,<\kappa)$, $\kappa=\omega_1\notin \mathfrak{Sp}^*$.
\end{lemma}
\begin{proof}
Let $T$ be a theory. As $T$ is countable, it is introduced by an initial segment of the collapse so we may assume without loss of generality that $T$ belongs to the ground model. Let $M$ be a model of $T$ of height $\kappa$ in the generic extension. Let $\dot{M}$ be a name. 

Let $j \colon \langle V_\kappa, \in, \dot{M}\rangle \to N$ be a weakly compact embedding. Then, $j$ lifts to the generic extension in the further extension by $\Col(\omega, [\kappa, <j(\kappa)))$. So, this further extension adds a model $j(\dot{M})$ which satisfies $T$ and $j\restriction M$ witnesses the lack of generic maximality of $M$.
\end{proof}
Note that if $\kappa$ is the least weakly compact in $L$, then $(\kappa^{+})^L$ belongs to $S_u^*$, and in particular to $S^*$. We can take $T$ to be the theory: $V = L$, there is a maximal cardinal and it is the unique weakly compact cardinal. For the same $T$, it is also the maximal element of $\mathfrak{Sp}^*_T$. Similar argument works for other local large cardinals such as ineffable cardinals, so it is suggestive to suspect that the consistency strength of making \emph{all} uncountable ordinals to avoid $S^*$ will be related to a non-local type of large cardinal. 

Recall that definition of virtually strong cardinal, due to Schindler, \cite{Schindler2000}, using the more recent terminology of virtual large cardinal, \cite{GitmanSchindler2018}.
\begin{definition}
A cardinal $\kappa$ is virtually strong (remarkable) if for every $\gamma \geq \kappa$ and for every $\delta \geq \kappa$, there is a transitive model $M$ containing $V_\delta$ such that in the generic extension by $\Col(\omega, V_\gamma)$, there is an elementary embedding $j \colon V_\gamma \to M$ with critical point $\kappa$, and $j(\kappa) \geq \delta$. 
\end{definition}
Virtually strong cardinals are consistent with $V = L$, and strictly weaker than $\omega$-Erd\H{o}s cardinals. 
\begin{theorem}\label{thm:remarkable-cardinal-making-all-ordinals-evasive}
Let $\kappa$ be a virtually strong cardinal. Then in the generic extension by $\Col(\omega, <\kappa)$, $S^* = \mathfrak{Sp}^* = \omega_1$.
\end{theorem}
\begin{proof}
Let $\dot{T}$ be a name for a theory in the generic extension. By the chain condition of the forcing, there is a bounded initial segment of the forcing adding it, so without loss of generality, we may assume that it is from the ground model.

Let $\dot{M}$ be a name for a model of $T$ of height $\alpha \geq \kappa$. By the definition of virtually strong cardinal, in the generic extension by $\Col(\omega, V_{\alpha})$ there is an elementary embedding $j \colon V_{\alpha} \to N$, for some $N \in V$, with critical point $\kappa$ and $j(\kappa) \geq \alpha + 1$. In particular, $j(\dot{M})$ is an $\Col(\omega, <j(\kappa))^N$-name for a transitive model of $T$, of height $j(\alpha) > \alpha$. This shows that $T$ cannot witness $\alpha \in S^*$, as the forcing $\Col(\omega,[\kappa,  <j(\kappa))$ introduces a transitive model of $T$ of height $> \alpha$. 

In order to show that $\alpha \notin \mathfrak{Sp}^*_T$ we must show that in the generic extension $j\image M \prec j(M)$. This follows from Silver's elementary embedding lifting theorem, \cite{CummingsHandbook}. Indeed, for every $\bar a \in M^{<\omega}$ and formula $\varphi$, if $p \Vdash \dot M \models \varphi(\bar a)$ then by elementarity $j(p) \Vdash j(\dot{M}) \models \varphi (j(\bar a))$, and as $j(p) = p$, the result follows.
\end{proof}
\begin{conjecture}
The consistency strength of the $\mathfrak{Sp}^* = \omega_1$ is exactly a virtually strong cardinal.
\end{conjecture}

As the following theorem shows, the strength of trivializing $S^*$ by itself is quite low.
\begin{theorem}
Let us assume that $\Ord$ is Mahlo and that Global Choice holds. Then, there is an inaccessible cardinal $\kappa$ such that $\Col(\omega, <\kappa)$ forces that $S^* = \omega_1$.
\end{theorem}
\begin{proof}
Let us start with a general lemma regarded a variant of Fodor's Lemma for our case. As was shown by Gitman, Hamkins and Karagila, \cite{GitmanHamkinsKaragila}, Fodor's lemma for classes might fail in general. In our case we care about the following weak instance.
\begin{lemma}\label{lemma:weak-fodor-lemma}
Let us assume that $\Ord$ is Mahlo. Then, for every class club $C$ and class function $F \colon C\cap \Inacc \to V$ such that $F(\alpha) \in V_\alpha$, there are $\delta_0 < \delta_1$ in $C$ such that $F(\delta_0) = F(\delta_1)$.
\end{lemma} 
\begin{proof}
Pick such $C$ and such $F$. Let $D$ be the class of all $\delta \in C$ such that for every $x \in V_\delta$, if there is $\gamma \in C$ such that $F(\gamma) = x$ then there is such $\gamma$ below $\delta$. Take $\delta_1 \in D$ inaccessible. Then $F(\delta_1) = x = F(\delta_0)$ for some $\delta_0 < \delta_1$.
\end{proof}
  
Let us assume towards a contradiction that the conclusion of the theorem fails. Then, for every $\kappa < \delta$ regular there is a $\Col(\omega, < \kappa)$-name $\dot{T}_\kappa$ for a theory and a $\Col(\omega,<\kappa)$-name for an ordinal $\dot{\gamma}_\kappa$ such that $\Vdash_{\Col(\omega,<\kappa)} \dot{\gamma}_\kappa \in S^* \setminus \omega_1$. 

By the chain condition of the Levy collapse, we may assume that $\dot{T}_\kappa \in V_\kappa$. We may bound the ordinal $\dot{\gamma}_\kappa$ by taking the supremum of all its possible values. Let $C$ be the class club of cardinals $\delta$ such that for every $\delta' < \delta$, $\Vdash \dot{\gamma}_{\delta'} < \delta$. By Lemma \ref{lemma:weak-fodor-lemma}, there are $\mu < \mu'$ in $C$ such that $\dot{T}_{\mu} = \dot{T}_{\mu'} = \dot{T}$, and by the construction of $C$, it is forced by $\Col(\omega, <\mu')$ that $\dot{\gamma}_{\mu}^{G\restriction \mu} < \dot{\gamma}_{\mu'}^G$.

By forcing with an initial segment, we may assume that $\dot{T}$ is in the ground model. 

Let $\mu < \mu'$ in $S$. Then, in the generic extension by $\Col(\omega, <\mu)$ there is a transitive model of $T$ of height $\gamma_\mu$, $M$. In the further generic extension by $\Col(\omega, [\mu, <\mu'))$, we add an additional transitive model $M'$ of $T$ of height $\gamma_{\mu'} > \gamma_\mu$, thus contradicting the assumption that $T$ witnesses $\gamma_\mu \in S^*$.
\end{proof}
In order to prove a partial converse to the theorem, let us commence with a couple of lemmas.
\begin{lemma}\label{lemma:omega1-inaccessible-to-reals}
For every real $r$, $\omega_1^{L[r]} \in S^*_u$. Moreover, if $\omega_1 = \omega_1^{L[r]}$ then for every $\alpha \geq \omega_1$ there is a proper forcing notion that forces $\alpha \in S^*_u$ and adds a theory $T$ such that $\alpha \in \mathfrak{Sp}^*_{T,u}$.
\end{lemma}
\begin{proof}
For the first part, let $T$ be a theory that codes $A = r$ and states that every ordinal is countable. 

Let us deal with the second part. Let us first collapse $|\alpha|$ to be $\omega_1$ using a $\sigma$-closed forcing, and let $B$ be a relation on $\omega_1$ coding $\alpha$ as an ordinal. Then, let $\mathcal{A}$ be an almost disjoint family of subsets of $\omega$ in $L[r]$ of size $\omega_1$, definable in $L_{\omega_1}[r]$. Using the almost disjoint coding, let $s$ be a real coding $B$, \cite{JensenSolovay}.

Let $T$ be the theory stating that $A = (r, s)$ and that for every ordinal $\beta$ there is a countable ordinal $\gamma$ such that $\otp \{\delta \in \omega_1 \mid (\delta, \gamma) \in B\}$ is $\beta$, as ordered by $B$. Note that this makes sense, as the correspondence between the ordinals and the initial segments of $B$ is unique.
\end{proof}

\begin{corollary}
If $\omega_1$ is not inaccessible in $L$, then it is a member of $S^*_u$. 
\end{corollary}
\begin{lemma}
Let us assume that the set of inaccessible cardinal in $L_{\omega_1^V}$ is bounded. Then $\omega_1 \in S_u^*$. Moreover, if there is a club $C \subseteq \omega_1$ which is $\Pi_1$-definable in $L[r]$ for some $r \subseteq \omega_1$, and avoids all inaccessible cardinals, then $\omega_1 \in S_u^*$.
\end{lemma}
\begin{proof}
Let us prove the moreover part, as it implies the first part. 

Let us assume, without loss of generality, that $\omega_1$ is inaccessible in the model $L[r]$ for every $r$.

Let $C$ be a club as in the statement of the lemma. Let us look at $M = L_{\omega_1}[r]$. The theory of $M$ includes the statement that no element of $C$ is an inaccessible and that $C$ is a club. Let $N$ be a transitive model in a generic extension which is elementary equivalent to $M$, and $N \cap \Ord > \omega_1^V$. Then, $N = L_\delta[r]$ for some $\delta$. As $C$ is $\Pi_1$-definable, $C^N \cap \omega_1^V = C$ and in particular, $\omega_1^V \in C^N$, which is a contradiction.   
\end{proof}
\begin{definition}
Let us say that $\kappa$ be $\alpha$-inaccessible if for every $\beta < \alpha$, the set of all $\beta$-inaccessible cardinals is unbounded at $\kappa$, and $\kappa$ is inaccessible.  
\end{definition}
Clearly, the set of all $\alpha$-inaccessible cardinals is $\Pi_1$-definable.

\begin{corollary}
The consistency strength of $\omega_1\notin S^*$ and $S^* = \omega_1$ is bounded from above by $\Ord$ is Mahlo. It is bounded from below by a cardinal $\kappa$ which is $\kappa$-inaccessible. 
\end{corollary}

\section{Erd\H{o}s cardinals}
In Lemma \ref{lemma:omega_1-not-in-Sp*}, we saw that the consistency strength of $\omega_1\notin \mathfrak{Sp}^*$ is exactly a weakly compact cardinal. The main theorem of this section is that certain large cardinal axioms outright imply that result. By standard arguments, the existence of a Woodin cardinal is sufficient, as one can force with Woodin's stationary tower. By replacing the stationary tower with a weaker ideal and analysing the ill foundedness of the obtained generic ultrapower, we obtain the following improvement.

\begin{theorem}\label{thm:erdos}
Let us assume that there is an $\omega_1$-Erd\H{o}s cardinal. Then, $\omega_1 \notin \mathfrak{Sp}^*$. 

More generally, if $\alpha$ is an uncountable ordinal and there is an $\alpha$-Erd\H{o}s cardinal, then $\alpha \notin \mathfrak{Sp}^*$.
\end{theorem}
\begin{proof}
The generalization is proved in exactly the same way as the particular case, so in order to simplify the notations we will prove the case of $\alpha = \omega_1$.

Let $\kappa$ be an $\omega_1$-Erd\H{o}s cardinal. So, almost by the definition, it means that the set
\[E=\{x \subseteq \kappa \mid \otp x = \omega_1\}\]
is weakly stationary in $P_{\omega_2}(\kappa)$, see \cite[Section 3]{ForemanHandbook} for the definition. Let $\mathbb{E}$ be the forcing notion of all stationary subsets of $E$, ordered by inclusion. The forcing $\mathbb{E}$ produces a generic $V$-ultrafilter, $\mathcal{U}$. Let us consider the obtained elementary embedding $j \colon V \to M$. 

Let us first derive a couple of well known properties of this elementary embedding.
\begin{lemma}
$\crit j = \omega_1$.
\end{lemma}
\begin{proof}
Let $f \colon P_{\omega_2}(\kappa)\to \alpha$, for $\alpha < \omega_1$. Let $E'$ be a condition. By the $\sigma$-completeness of the club filter, there is $E'' \subseteq E'$ such that $f \restriction E''$ is constant, and in particular, $[f] = j(\gamma)$ for some $\gamma < \omega_1$. 
\end{proof}
\begin{lemma}
$j(\omega_1) = \kappa$. In particular, $M$ is well founded up to $j(\omega_1)$.
\end{lemma}
\begin{proof}
For every $\alpha \leq \kappa$, let $f_\alpha(x) = \otp(x \cap \alpha)$. We would like to show that $[f_\alpha] = \alpha$ for all $\alpha \leq \kappa$

Indeed, let $g$ be a function such that $[g] < [f_\kappa] = [c_{\omega_1}]$ and let $E'$ be a condition forcing that. Let us define $h(x) = \rho_x$ such that $\otp (x \cap \rho_x) = g(x)$. This is possible, as almost everywhere $g(x) < \omega_1$. The function $h$ is a regressive function on the stationary set, $E'$, and therefore there is $E'' \subseteq E'$ on which it is constant. Let $\rho$ be the fixed value that $h$ obtained. Then, $E'' \Vdash [f_\rho] = [g]$.

We conclude that every ordinal below $j(\omega_1)$ is of the form $[f_\alpha]$ for some $\alpha < \kappa$. Clearly, if $\alpha < \beta$ then $[f_\alpha] < [f_\beta]$ (as club many $x$ contains $\alpha$ as a member), so the lemma follows. 
\end{proof}
Even though the following lemma is not required for this argument, we include it here as it completes the picture of the structure of the generic ultrapower.
\begin{lemma}
$j\image \kappa \in M$. 
\end{lemma}
\begin{proof}
The identity function $id \colon P_{\omega_2} \kappa \to P_{\omega_2} \kappa$ represents $j\image \kappa$: Indeed, $j\image \kappa \subseteq [id]$ as every $\alpha < \kappa$ belongs to almost all $x \in P_{\omega_2}(\kappa)$. 

On the other hand, if $E' \Vdash [g] \in [id]$, then $g$ is regressive on $E$, and thus constant on some $E'' \subseteq E'$. 
\end{proof}

We are now ready for the proof of the theorem. Let $L_{\omega_1}[A]$ be a model of some theory $T$. Then, in the generic extension by $\mathbb E$, we can apply the elementary embedding and obtain $j\restriction L_{\omega_1}[A] \colon L_{\omega_1}[A] \to j(L_{\omega_1}[A])$ in the generic extension. Moreover, as $j(\omega_1)$ is well founded, $j(L_{\omega_1}[A])$ is a transitive model and $j(\omega_1) > \omega_1$, as wanted.
\end{proof}
\section{Evasive ordinals}
In the previous section, we discussed several instances of ordinals that belong to the classes $S, S^*, S^*_u$ and so on. As we showed, the consistency strength of making those classes as small as possible is below $0^{\#}$. In this section, we discuss a stronger form of avoiding those classes. 
\begin{definition}
Let $R\in \{S, S_u, S^*, S^*_u, \dots\}$. 

An ordinal $\alpha$ is $R$-evasive if for every set forcing notion $\mathbb{P}$, 
\[\Vdash_{\mathbb{P}} \check\alpha < \dot{\omega}_1 \vee \check\alpha \notin \dot{R}.\]
\end{definition}

\begin{lemma}\label{lemma:S-non-evasive}
Every cardinal is not $S$-evasive.
\end{lemma}
\begin{proof}
For regular cardinals, there is a forcing notion making them into $\aleph_1$. For singular cardinals, there is a forcing notion making in into $\aleph_\omega$.
\end{proof}
\begin{question}Is there an $S$-evasive ordinal?
\end{question}

On the other hand, it is clear that a proper class of Woodin cardinals entails that every uncountable cardinal is $R$-evasive for $R\in \{S^*, \mathfrak{Sp}^*\}$. In Theorem \ref{thm:erdos}, we provide a similar result, assuming a weaker large cardinal hypothesis.

\begin{lemma}
Let $X$ be a set of ordinals. If $X^{\#}$ does not exists, then the class of $S^{*}_u$-evasive ordinals and $\mathfrak{Sp}^*_u$-evasive ordinals is bounded. 
\end{lemma}
\begin{proof}
First, since $X^{\#}$ does not exists, there is a cardinal $\mu$ which is regular in $V$ and successor in $L[X]$. Indeed, pick some large enough singular cardinal in $V$. Then, its successor is computed correctly in $L[X]$. Let $\nu$ be the predecessor of $\mu$ in $L[X]$.

Let $G \subseteq \Col(\omega, \nu)$ be $V$-generic. In $V[G]$, there is a real $r$ such that $\mu = \omega_1^{L[r]}$. By Lemma \ref{lemma:omega1-inaccessible-to-reals}, we conclude that $\mu \in S^*_u$ and in $\mathfrak{Sp}^*_u$ in $V[G]$, and moreover, every ordinal $\alpha \geq \mu$ belongs to those classes in some generic extension by an additional proper forcing, by collapsing it to have cardinality $\aleph_1$ and coding its order type using almost disjoint coding.
\end{proof}

\begin{theorem}
Let us assume that for every $\alpha$ there is an $\alpha$-Erd\H{o}s cardinal. Then, every ordinal is $(\mathfrak{Sp}^*\cup S^*)$-evasive.
\end{theorem}
\begin{proof}
Let $\alpha$ be an ordinal and let us assume that there is a generic extension by a forcing notion $\mathbb{P}$ in which $\alpha$ is uncountable and belongs to $\mathfrak{Sp}^*_T$ for some $T$. 

In the generic extension there is an $\alpha$-Erd\H{o}s cardinal, as there was a $\beta$-Erd\H{o}s cardinal in the ground model for $\beta$ larger than the size of the forcing $\mathbb{P}$ and $\alpha$. Now let us apply Theorem \ref{thm:erdos} and derive a contradiction.
\end{proof}
The theorems of this section bound the consistency strength of the assertion "every ordinal is $\mathfrak{Sp}^*$-evasive" between the existence of sharps for every set and the existence of Erd\H{o}s cardinals.
\begin{question}
What is the consistency strength of the assertion "every ordinal $\alpha$ is $\mathfrak{Sp}^*$-evasive?
\end{question}

\providecommand{\bysame}{\leavevmode\hbox to3em{\hrulefill}\thinspace}
\providecommand{\MR}{\relax\ifhmode\unskip\space\fi MR }
\providecommand{\MRhref}[2]{%
  \href{http://www.ams.org/mathscinet-getitem?mr=#1}{#2}
}
\providecommand{\href}[2]{#2}

\end{document}